\documentclass[12pt]{amsart}
\usepackage{latexsym}
\usepackage{amsfonts}
\usepackage{amsthm}
\usepackage{amsmath}
\usepackage{amssymb}
\usepackage{hyperref}

\numberwithin{equation}{section}

\newtheorem{theorem}{Theorem}
\newtheorem{corollary}[theorem]{Corollary}
\newtheorem{proposition}[theorem]{Proposition}
\newtheorem{claim}[theorem]{Claim}

\theoremstyle{definition}

\newtheorem{remark}[theorem]{Remark}
\newtheorem{example}[theorem]{Example}

\newcommand{\cupp}{\mathbin{\smallsmile}}

\newcommand{\Z}{\mathbb{Z}}
\newcommand{\C}{\mathbb{C}}

\DeclareMathOperator{\GL}{GL}
\DeclareMathOperator{\SO}{SO}
\DeclareMathOperator{\Sp}{Sp}

\begin{document}

\title{Deformed Cohomology of Flag Varieties}
\author{Oliver Pechenik}
\author{Dominic Searles}
\address{Department of Mathematics, University of Illinois at Urbana-Champaign, Urbana, IL 61802}
\email{pecheni2@illinois.edu, searles2@illinois.edu} 
\date{\today}

\begin{abstract}
This paper introduces a two-parameter deformation of the cohomology of generalized flag varieties. One special case is the Belkale-Kumar deformation (used to study eigencones of Lie groups). Another picks out intersections of Schubert varieties that behave nicely under projections. Our construction yields a new proof that the Belkale-Kumar product is well-defined. This proof is shorter and more elementary than earlier proofs.
\end{abstract}

\maketitle 

\section{Introduction}
In 2006, P.~Belkale and S.~Kumar \cite{belkale.kumar} introduced a new product structure on the cohomology of generalized flag varieties. 
They used this deformed product to obtain a maximally efficient solution to the Horn problem in general Lie type (generalizing the famous Horn problem on eigenvalues of sums of Hermitian matrices). The irredundancy of this solution was proved in 2010 by N.~Ressayre \cite{ressayre}.
More recently this product has been used to further study eigencones of compact connected Lie groups \cite{BelkaleKumar10, Ressayre} and the representation theory of (semisimple parts of) Levi subgroups \cite{BelkaleKumarRessayre}. 

This paper introduces a more general product that has the Belkale-Kumar product as a specialization. Another specialization identifies intersections of Schubert varieties with nice projection properties. From our general construction, we obtain a new and significantly easier proof that the Belkale-Kumar product is well-defined. 

Let $G$ be a complex connected reductive Lie group. Choose Borel and opposite Borel subgroups $B, B_-$ and maximal torus $T = B \cap B_-$. Let $W$ denote the Weyl group $N_G(T) / T$. For $w\in W$, we denote the Coxeter length of $w$ by $l(w)$. 
Fix a parabolic subgroup $B \subseteq P\subset G$. 
Let $W_P$ denote the associated parabolic subgroup of $W$, and $W^P$ denote the set of minimal length coset representatives of $W/W_P$. For $w \in W^P$, the \emph{Schubert variety} $X_w=\overline{B_-wP/P} \subseteq G/P$ has codimension $l(w)$. The Poincar\'e duals $\{ \sigma_w \}$ of the Schubert varieties form an additive basis of the cohomology ring $H^\star(G/P)$. That is, 
\[ \sigma_u \cupp \sigma_v = \sum_w c_{u,v}^w \sigma_w,\] where $c_{u,v}^w \in \Z_{\geq 0}$ is a \emph{Schubert structure constant}. (In the case $G=\GL_n(\C)$ and $P$ is maximal, these structure constants are the Littlewood-Richardson coefficients). Let $w^\vee=w_0ww_0^P$, where $w_0, w_0^P$ are the longest elements of $W, W_P$, respectively. The number $c_{u,v}^w$ is nonzero exactly when generic translates of $X_u, X_v,$ $X_{w^\vee}$ intersect in a finite nonzero number of points; in that case, $c_{u,v}^w$ counts the number of such points. 

For each simple root $\alpha$ associated to $P$ (so that $\alpha$ is inverted by some element of $W^P$), we introduce a complex variable $t_\alpha$ and a positive real variable $s_\alpha$. For a positive root $\beta$, let $n_{\alpha\beta}$ denote the multiplicity of $\alpha$ in the simple root expansion of $\beta$ and define \[t^\beta = \prod_\alpha t_\alpha^{{n_{\alpha\beta}^{s_\alpha}}},\] where the product is over the simple roots associated to $P$. Then define $F_w(t,s)$ to be the product of the $t^\beta$ over all positive roots $\beta$ that are inverted by $w$. We define a product on $H^\star(G/P)$ by
\begin{equation}\label{eq:product_def}
 \sigma_u \star_{t,s} \sigma_v = \sum_w \frac{F_w(t,s)}{F_u(t,s)F_v(t,s)} c_{u,v}^w \sigma_w.
\end{equation}
We recover the Belkale-Kumar product $\odot_t$ as the specialization $\star_{t,1}$. (This is immediate from the description of $\odot_t$ in \cite{evens.graham}.) Most interest has been in the further specialization $\odot_0 = \star_{0,1}$ given by evaluating each $t_\alpha$ to $0$.

\begin{theorem}\label{thm:main}
The product $\star_{t,s}$ is a well-defined commutative associative product. In particular, $F_u(t,s)F_v(t,s)$ divides $F_w(t,s)$ whenever the Schubert structure constant $c_{u,v}^w$ is nonzero.
\end{theorem}

\begin{corollary}[\cite{belkale.kumar,evens.graham}]\label{cor:bk}
The Belkale-Kumar product is well-defined.
\end{corollary}
\begin{proof}
This follows from Theorem~\ref{thm:main} by $\odot_t = \star_{t,1}$.
\end{proof}

Corollary~\ref{cor:bk} was proved by P.~Belkale--S.~Kumar \cite{belkale.kumar} using geometric invariant theory (specifically a Hilbert-Mumford criterion for semistability) and by S.~Evens--W.~Graham \cite{evens.graham} using relative Lie algebra cohomology. In contrast, our proof of Theorem~\ref{thm:main} (and hence of Corollary~\ref{cor:bk}) uses only straightforward analysis of the tangent spaces to Schubert varieties.

This paper is structured as follows. In Section~\ref{sec:main_proof}, we prove Theorem~\ref{thm:main}.  In Section~\ref{sec:star_product} we study the limit of $\star_{t,s}$ as $s \to 0$, and describe its geometric significance. As a corollary, we obtain an independent and completely elementary proof of Corollary~\ref{cor:bk} in the case $G = \GL_n(\C)$. 

\section{Proof of Theorem~\ref{thm:main}}
\label{sec:main_proof}
Commutativity is clear.
For associativity, observe that
\begin{align*}
(\sigma_u \star_{t,s} \sigma_v) \star_{t,s} \sigma_w &= \sigma_w \star_{t,s} \sum_x \frac{F_x}{F_uF_v} c_{u,v}^x \sigma_x \\
&= \sum_{x,y} \frac{F_x}{F_uF_v} \frac{F_y}{F_wF_x} c_{u,v}^x c_{w,x}^y \sigma_y \\
&= \sum_{x,y} \frac{F_y}{F_uF_v F_w} c_{u,v}^x c_{w,x}^y \sigma_y,
\end{align*} while similarly 
\[
\sigma_u \star_{t,s} (\sigma_v \star_{t,s} \sigma_w) = \sum_{x,y} \frac{F_y}{F_uF_v F_w} c_{v,w}^x c_{u,x}^y \sigma_y.
\] 
Associativity then follows immediately from that of the ordinary cup product.

We now prove $\star_{t,s}$ is well-defined. Let $w_1, w_2, w_3 \in W^P$. Then  $c_{w_1, w_2}^{w_3^\vee}(G/P) = c_{w_1, w_2}^{w_3^\vee}(G/B)$. In particular, since $c_{w_1, w_2}^{w_3^\vee}(G/P) \neq 0$ implies $c_{w_1, w_2}^{w_3^\vee}(G/B) \neq 0$, it suffices to assume that $c_{w_1, w_2}^{w_3^\vee}(G/B) \neq 0$ and to show that $F_{w_1}F_{w_2}$ divides $F_{w_3^\vee}$.

Most of the facts described below are well-known. We learned some of these ideas from \cite{richmond:thesis}, where they appear with further details. Our proof is heavily indebted to \cite[Theorem~1]{Purbhoo}, in particular for the key idea that filters give rise to $B$-stable subspaces.

\begin{claim}\label{claim:tangent_stuff}
For generic $b_i \in B$ \begin{equation*}\label{eq}
T_{eB}(G/B)=\bigoplus_{i=1}^3 b_i \cdot \frac{T_{eB}(G/B)}{T_{eB}(w_i^{-1}X_{w_i})}.
\end{equation*}
\end{claim}
\begin{proof}
By Kleiman transversality \cite{kleiman}, $c_{w_1, w_2}^{w_3^\vee}(G/B)\neq 0$ implies that the intersection $\bigcap_{i=1}^3 g_iX_{w_i}$ is transverse and nonempty for generic $g_i\in G$. Therefore 
$\bigcap_{i=1}^3 b_iw_i^{-1}X_{w_i}$ is transverse at $eB$ for generic $b_i\in B$. 
Since $T_{eB}(b_iw_i^{-1}X_{w_i})=b_i \cdot T_{eB}(w_i^{-1}X_{w_i})$,
the claim follows. 
\end{proof}

Fix generic $b_i \in B$ and let
\[I_i = \frac{T_{eB}(G/B)}{T_{eB}(w_i^{-1}X_{w_i})}.\]
Let $\Phi = \Phi^+ \sqcup \Phi^-$ denote the standard partition of the roots of $G$ into positives and negatives (so that the positive root spaces correspond to infinitesimal curves through the Borel). We will use the natural poset structure on $\Phi^+$, that is $\beta \leq \gamma$ if and only if $\gamma - \beta$ is a nonnegative integral combination of positive roots. 
By applying the Cartan involution to identify $\mathfrak{b}$ with $\mathfrak{b}_-$, we have
\[T_{eB}(G/B)= \bigoplus_{\beta\in \Phi^+}\mathfrak{g}_{\beta} \qquad \mbox{and} \qquad I_i = \bigoplus_{\beta\in \Phi^+ \cap w_i^{-1}\Phi^-}\mathfrak{g}_{\beta},\] where $\mathfrak{g}_\beta$ denotes the root space corresponding to the root $\beta$.

Recall that a \emph{filter} (or \emph{upset}) of a poset is a subset $\mathcal{J}$ such that if $x \in \mathcal{J}$ and $x \leq y$, then $y \in \mathcal{J}$. For $\mathcal{J}$ a filter in $\Phi^+$, let $J=\bigoplus_{\beta\in \mathcal{J}}\mathfrak{g}_{\beta}\subset T_{eB}(G/B)$. Since $\mathcal{J}$ is a filter, $b_i \cdot J=J$. 

Let $|w_i|_\mathcal{J}$ denote the number of $\beta \in \mathcal{J}$ such that $w_i\beta \in \Phi^-$.  Suppose $|w_1|_\mathcal{J} + |w_2|_\mathcal{J} + |w_3|_\mathcal{J} > |\mathcal{J}|$. Then 
\begin{align*}
\dim \, (b_1 \cdot I_1 + b_2 \cdot I_2 + b_3 \cdot I_3) / J  &\leq \dim \, b_1 \cdot I_1/J + b_2 \cdot I_2 / J + b_3 \cdot I_3 / J \\
&\leq \dim \, I_1/J + I_2 / J + I_3 / J \\
&\leq \sum_{i=1}^3 \dim I_i - \sum_{i=1}^3 \dim I_i \cap J\\
&< \dim T_{eB}(G/B) - \dim J \\
&=\dim T_{eB}(G/B)/ J ,
\end{align*}
so by Claim~\ref{claim:tangent_stuff} we must have $|w_1|_\mathcal{J} + |w_2|_\mathcal{J} + |w_3|_\mathcal{J} \le |\mathcal{J}|$, i.e., $|w_1|_\mathcal{J}+|w_2|_\mathcal{J} \le |w_3^{\vee}|_\mathcal{J}$.

Fix a simple root $\alpha$. Let $\mathcal{J}_{\alpha, k}$ be the set of roots of $\Phi^+$ that use $\alpha$ at least $k$ times in their expansion into simple roots. Each $\mathcal{J}_{\alpha, k}$ is a filter in $\Phi^+$.
By the above, we have $|w_1|_{\mathcal{J}_{\alpha, k}} + |w_2|_{\mathcal{J}_{\alpha, k}} \leq |w_3^\vee|_{\mathcal{J}_{\alpha, k}}$ for all $\alpha$ and all $k$.
Hence the degree of $t_\alpha$ in $F_{w_3^\vee}$ is at least the degree of $t_\alpha$ in $F_{w_1}F_{w_2}$, so $F_{w_1}F_{w_2}$ divides $F_{w_3^\vee}$. \qed

\section{The limit $s \to 0$}
\label{sec:star_product}
We write $\star_t$ for the limit of $\star_{t,s}$ as $s \to 0$. In this section, we give an independent and elementary proof that $\star_t$ is a well-defined associative and commutative product. We then interpret $\star_0$ geometrically.

Let $S_w(t)$ denote the limit of $F_w(t,s)$ as $s \to 0$. 
The product $\star_t$ on $H^\star(G/P)$ may then be defined by replacing each $F(t,s)$ by $S(t)$ in Equation~\ref{eq:product_def}.

For $G = \GL_n(\C)$, we will show that $\star_t$ coincides with the Belkale-Kumar product $\odot_t$, while for maximal parabolics in general type $\star_t$ coincides instead with the ordinary cup product. In general it is distinct from both.

\begin{theorem}\label{thm:star_t}
$\star_t$ is a well-defined commutative associative product.
\end{theorem}
\begin{proof}
Commutativity is clear, while associativity is proved exactly as in the proof of Theorem~\ref{thm:main}.

We now prove $\star_t$ is well-defined. Let $P_\alpha$ denote the maximal parabolic subgroup of $G$ associated to a simple root $\alpha$ of $P$. Define a projection $\pi_\alpha:G/P \rightarrow G/P_\alpha$ by $\pi_\alpha(gP) = gP_\alpha$.
Observe that
$\pi_\alpha$ is $G$-equivariant. 

For $w \in W^P$, let $w_\alpha$ denote the minimal length coset representative of $wW_{P_\alpha}$. Then $\pi_\alpha$ maps points of $X_w$ to $X_{w_\alpha}$.

\begin{claim}\label{claim:region_count}
If $c_{u,v}^w(G/P) \neq 0$, then for each $\alpha$, $l(u_\alpha)+l(v_\alpha) \leq l(w_\alpha)$.
\end{claim}
\begin{proof}
If $l(u_\alpha)+l(v_\alpha)>l(w_\alpha)$, then for dimension reasons, generic translates of $X_{u_\alpha}$, $X_{v_\alpha}$, $X_{(w^{\vee})_\alpha}$ have empty intersection in $G/P_\alpha$.

Since $c_{u,v}^w \neq 0$, for generic $(g_1, g_2, g_3) \in G^3$, there is a point $gP \in g_1 X_u \cap g_2 X_v \cap g_3 X_{w^{\vee}} \subseteq G/P$. 
This implies $\pi_\alpha(gP) \in g_1 X_{u_\alpha} \cap g_2 X_{v_\alpha} \cap g_3 X_{(w^{\vee})_\alpha} \subseteq G/P_\alpha$. In particular this latter intersection is nonempty, so $l(u_\alpha)+l(v_\alpha)\le l(w_\alpha)$.
\end{proof}

The degree of $t_\alpha$ in $S_w(t)$ is exactly the number of positive roots $\beta$ inverted by $w$ that use $\alpha$ in their simple root expansion. This number is $l(w_\alpha)$. Therefore, the degree of $t_\alpha$ in $\frac{S_w(t)}{S_u(t)S_v(t)}$ is $l(w_\alpha)-l(u_\alpha)-l(v_\alpha)$.

Let $u,v,w\in W^P$ with $c_{u,v}^w(G/P) \neq 0$. Then by Claim~\ref{claim:region_count}, $l(w_\alpha)-l(u_\alpha)-l(v_\alpha)\ge 0$ for all $\alpha$, and so $S_u(t)S_v(t)$ divides $S_w(t)$ as desired.
\end{proof}

As a corollary, we obtain the following special case of Corollary~\ref{cor:bk}.
\begin{corollary}
For $G = \GL_n(\C)$, the Belkale-Kumar product $\odot_t$ is well-defined. 
\end{corollary}
\begin{proof}
For $\GL_n(\C)$, we always have $n_{\alpha\beta} \leq 1$, so $F(t,s) = S(t)$ and $\odot_t = \star_{t,s}$.
\end{proof}

For any $Q \supset P$ and $w \in W^P$, there is a unique parabolic decomposition $w = w'w''$, where $w' \in W^Q$ and $w'' \in W^P \cap W_Q$. 
Suppose $c_{u,v}^{w^\vee} \neq 0$. We say that the triple $(u, v, w) \in (W^P)^3$ is \emph{$Q$-factoring} if $g_1 X_{u'} \cap g_2 X_{v'} \cap  g_3 X_{{w'}}$ is a finite (nonempty) set of points for generic $g_i \in G$, or equivalently if $g_1 X_{u''} \cap g_2 X_{v''} \cap  g_3 X_{w''}$ is generically a finite (nonempty) set of points.

Let $a_{u, v}^w :=  \frac{S_w(0)}{S_u(0)S_v(0)} c_{u,v}^w$ denote the structure constants of the ring $(H^\star(G/P), \star_0)$. 
\begin{proposition}
\[a_{u,v}^w = \begin{cases}
              c_{u, v}^w & \mbox{if $(u,v,w^\vee)$ is $Q$-factoring for every $Q \supset P$,}\\
              0 & \mbox{otherwise.}
              \end{cases}\]
\end{proposition}
\begin{proof}
This is trivial if $c_{u, v}^w=0$, so assume it is positive. Suppose $(u, v, w^\vee)$ is not $Q$-factoring for some $Q \supset P$. We may assume that $Q$ is a maximal parabolic $P_\alpha$ for some simple root $\alpha$. Then $l({w'}) > l(u') + l(v')$. Therefore $t_\alpha$ has positive degree in $\frac{S_w(t)}{S_u(t)S_v(t)}$, whence $\frac{S_w(0)}{S_u(0)S_v(0)} = 0$.
\end{proof}

\begin{remark}
Triples $(u,v,w)$ that are $P_\alpha$-factoring for some fixed collection of maximal parabolics $P_\alpha$ may be picked out by taking the limit of $\star_{t,s}$ as the corresponding $s_\alpha \to 0$, and then setting $t=0$ and other $s_\alpha = 1$.  
\end{remark}
              
\begin{remark}
By \cite[Theorem~1.1]{Richmond}, the numbers $a_{u,v}^w$ factor as  $c_{u', v'}^{w'} c_{u'', v''}^{w''}$. Iterating this factorization for every maximal $P_\alpha \supset P$, we obtain a factorization of $a_{u,v}^w$ as a product of Schubert structure constants $c_{x,y}^z$ on maximal parabolic quotients $G/P_\alpha$.

Richmond \cite{Richmond} also notes that $(u, v, w)$ is $Q$-factoring for each $Q \supset P$ when $(u, v, w)$ is Levi-movable in the sense of \cite[Definition~4]{belkale.kumar}. Therefore $\star_0$ may be thought of as `less-degenerate' than $\odot_0$, since a generally smaller collection of Schubert structure constants are set to $0$.
\end{remark}

\begin{example}
Let $G = \SO_9(\C)$ and $P$ be the parabolic associated to the second and fourth simple roots (where the fourth is the short root). Of the 8271 nonzero Schubert structure constants for $H^\star(G/P)$, 807 are nonzero for the deformation $\star_0$. Of these only 597 represent Levi-movable triples and so are nonzero in the Belkale-Kumar deformation $\odot_0$. An example of one of the 210 nonzero $a_{u,v}^w$ coefficients not coming from a Levi-movable triple is $a_{1324, 1\overline{2}34}^{3\overline{2}14}=1$. (Here we identify $W$ with the group of signed permutations on four letters).

Of the 193116 nonzero Schubert structure constants for $H^\star(G/B)$, only 2439 are nonzero for $\star_0$. Of these, 2103 arise from Levi-movable triples.
\qed
\end{example}

\begin{example}
Let $G = \Sp_{12}(\C)$ and $P$ be the parabolic associated to the fourth simple root (where the sixth is the long root). There are $99105$ nonzero Schubert structure constants for $H^\star(G/P)$. Since $P$ is maximal, these are all nonzero for the deformation $\star_0$. However only $7962$ are nonzero for the Belkale-Kumar deformation $\odot_0$. \qed
\end{example}

\section*{Acknowledgments}
OP was supported by an Illinois Distinguished Fellowship, an NSF Graduate Research Fellowship, and NSF MCTP
grant DMS 0838434.
DS was supported by a Golub Research Assistantship and a University of Illinois Dissertation Completion Fellowship.

This project was inspired by a talk of Sam Evens at the University of Illinois at Urbana-Champaign in April 2014. We are grateful for very helpful conversations with Sam Evens, William Haboush, Edward Richmond, and Alexander Yong.


\begin{thebibliography}{9999999999}
\bibitem[BeKu06]{belkale.kumar} P.~Belkale and S.~Kumar, \emph{Eigenvalue problem and a new product in cohomology of flag varieties}, Invent. Math. {\bf 166.1} (2006), 185--228.
\bibitem[BeKu10]{BelkaleKumar10} P.~Belkale and S.~Kumar, \emph{Eigencone, saturation and Horn problems for symplectic and odd orthogonal groups}, J. Algebraic Geom. {\bf 19} (2010), 199--242.
\bibitem[BeKuRe12]{BelkaleKumarRessayre} P.~Belkale, S.~Kumar and N.~Ressayre, \emph{A generalization of Fulton's conjecture for arbitrary groups}, Math. Ann. {\bf 354.2} (2012), 401--425.
\bibitem[EvGr13]{evens.graham} S.~Evens and W.~Graham, \emph{The Belkale-Kumar cup product and relative Lie algebra cohomology}, Int. Math. Res. Not. (2013), 1901--1933.
\bibitem[Kl74]{kleiman} S. Kleiman, \emph{The transversality of a general translate}, Compos. Math.
{\bf 28} (1974), 287--297.
\bibitem[Pu06]{Purbhoo} K.~Purbhoo, \emph{Vanishing and non-vanishing criteria in Schubert calculus}, Int. Math. Res. Not. (2006), 1--38.
\bibitem[Re10]{ressayre} N.~Ressayre, \emph{Geometric invariant theory and the generalized eigenvalue problem}, Invent.\ Math. {\bf 180.2} (2010), 389--441.
\bibitem[Re12]{Ressayre} N.~Ressayre, \emph{A cohomology-free description of eigencones in types A, B, and C}, Int. Math. Res. Not. (2012), 4966--5005.
\bibitem[Ri08]{richmond:thesis} E.~Richmond,  \emph{Recursive structures in the cohomology of flag
varieties}, PhD thesis, University of North Carolina at Chapel Hill, 2008.
\bibitem[Ri12]{Richmond} E.~Richmond, \emph{A multiplicative formula for structure constants in the cohomology of flag varieties}, Michigan Math. J. {\bf 61.1} (2012), 3--17.
\end{thebibliography}
\end{document}